\documentclass[12pt]{article}
\usepackage{amssymb,latexsym,amsfonts,amsmath,amsthm}
\usepackage{graphics,graphicx,epstopdf}
\usepackage[colorlinks=true,linktocpage=true]{hyperref}

\newtheorem{theorem}{Theorem}[section]
\newtheorem{lemma}[theorem]{Lemma}
\newtheorem{proposition}[theorem]{Proposition}
\newtheorem{definition}{Definition}[section]
\newtheorem{notation}{Notation}[section]
\newtheorem{remark}{Remark}[section]
\newcommand{\N}{\mathbb{N}}

\newcommand{\R}{\mathbb{R}}

\newcommand{\1}{\mathbf{1}}

\newcommand{\Exp}[1]{ \mathbf{E}\left(\hspace{1mm} #1 \hspace{1mm}\right) }
\newcommand{\Prob}[1]{ \mathbf{P}\left(\hspace{1mm} #1 \hspace{1mm}\right) }
\newcommand{\scattering}{\mathcal{S}}

\begin{document}
\title{Stochastic Integral Operator Model \\ for IS, US and WSSUS Channels}
\author{Onur Oktay}
\date{ }
\maketitle

\begin{abstract}
In this article, we proved that, under weak and natural requirements, uncorrelated scattering (in particular WSSUS) channels can be modeled as stochastic integrals. Moreover, if we assume (not only uncorrelated but also) independent scattering, then the stochastic integral kernel is an additive stochastic process. This allows us to decompose an IS channel into a sum of independent channels; one deterministic, one with a Gaussian kernel, and two others described by the Levy measure of the additive process.
\end{abstract}

%

\section{Introduction}

A channel is the environment between the transmitter and the receiver. It is usually modeled as a linear operator $H$. The modifications and disturbances of the signal through the channel is embedded in this operator. 

A linear time variant channel has the operator representation $$Hf(t) = \int h(t,u)f(t-u)\ du$$ whether $h(t,u)$ is an actual function or a symbol. $h$ is the time variant impulse response of the channel with the autocorrelation
$$R_h(t,s,u,v) = \Exp{ h(t,u)h(s,v) } - \Exp{h(t,u)}\Exp{h(s,v)}$$
where $R_h$ is a positive definite function or symbolic impulse response of the bilinear operator $R$ defined by $$R(f,g)(t,s) = \Exp{Hf(t)Hg(s)} - \Exp{Hf(t)}\Exp{Hg(s)}$$ with the operator representation $$R(f,g)(t,s) = \iint R_h(t,s,u,v) f(t-u)g(s-v)\ dudv$$ 
The WSSUS channel assumption in the literature is the following: $\Exp{h(t,u)} = 0$ 
and
\begin{eqnarray}\label{def:WSSUS}
R_h(t,s,u,v) = P(t-s,u)\delta(u-v).
\end{eqnarray}
so that $\Exp{Hf(t)} = 0$ 
and
\begin{eqnarray}\label{WSSUS_op}
\Exp{ Hf(t)Hg(s) } = \int P(t-s,u) f(t-u)g(s-u)\ du.
\end{eqnarray}
A channel has independent scattering (IS) property if whenever $u\neq v$, $h(t,u)$ and $h(s,v)$ are independent random variables. IS implies the uncorrelated scattering (US) property
\begin{eqnarray}\label{def:US}
R_h(t,s,u,v) = P(t,s,u)\delta(u-v).
\end{eqnarray}
If a channel has the IS property, then $Hf$ and $Hg$ are independent processes whenever \\ $supp(f)\cap supp(g) = \emptyset$. We call the latter the {\it weak-IS property}. Similarly, US property \eqref{def:US} implies
\begin{eqnarray}\label{US_op}
\Exp{Hf(t)Hg(t)} = 0.
\end{eqnarray}
for each pair of signals $f,g$ for which $supp(f)\cap supp(g) = \emptyset$. We call \eqref{US_op} the {\it weak-US property}.

In this paper, we assume that $H$ is defined on the indicator functions of bounded intervals. This is a realistic assumption since a channel can be tested/sounded with finite time duration bang-bang (switch on/off) signals in applications. For simplicity we omitted the carrier frequency in our discussion without the loss of generality. 
We also assume that $H$ satisfies the following continuity property: For each $t\in\R$,
\begin{eqnarray}\label{Cont_mean2_op}
\lim_{|b-a|\to 0} \Exp{ |H\1_{[a,b)}(t)|^2 } = 0.
\end{eqnarray}

In Section~\ref{sec:SIO_rep}, we show that $H$ has a stochastic integral operator (SIO) representation if it satisfies \eqref{US_op} and \eqref{Cont_mean2_op}. We define the impulse response and spreading functions as stochastic processes and derive related SIO representations of $H$. In Section~\ref{sec:WSSUS}, we derive what the IS, US and WSSUS properties translate into for a SIO. In Section~\ref{sec:H_decomp}, we show that if $H$ has the weak-IS property, then it has a decomposition as a sum of four independent operators, each of which capture a distinct characteristic of the channel.

\section{SIO representation of weak-US channels}\label{sec:SIO_rep}

We begin this section with a definition. We shall construct each of those and prove the equalities in the coming subsections.

\begin{definition}\normalfont
We call the stochastic processes related to the operator $H$ as below.
\begin{enumerate}
\item $X(t,u)$ is the integrated kernel and $X(t,du)$ the kernel symbol, where $$Hf(t) = \int f(u)X(t,du) $$
\item $Y(t,u)$ is the integrated impulse response and $Y(t,du)$ the impulse response symbol, where $$Hf(t) = \int f(t-u)Y(t,du) $$
\item $\sigma(t,\xi)$ is the Kohn-Nirenberg symbol, where $$Hf(t) = \int e^{2\pi it\xi} \widehat{f}(\xi) \sigma(t,\xi)\ d\xi$$ 
\item $\eta(u,\gamma)$ is the integrated spreading symbol and $\eta(du,\gamma)$ the spreading symbol, where $$Hf(t) = \iint e^{2\pi it\gamma} f(t-u) \eta(du,\gamma) d\gamma $$
\end{enumerate}
\end{definition}

\subsection{Stochastic kernel}\label{sec:SIO_rep_X}

We define a random process $X$ as follows: For every $t\in\R$, 
$$X(t,u) = \left\{\begin{array}{cl}
H\1_{[0,u)}(t) &; u>0 \\
0 &; u=0 \\
-H\1_{[u,0)}(t)  &; u<0 \\ 
\end{array}\right.$$
For every $t,u\in\R$, $\Exp{X(t,u)}=0$.
By \eqref{US_op}, we have
\begin{eqnarray}\label{US_X}
\Exp{ (X(t,u_1) - X(t,u_2)) (X(t,u_3) - X(t,u_4))} = 0.
\end{eqnarray}
 for all $t\in\R$ and $u_1>u_2\geq u_3>u_4$. By \eqref{Cont_mean2_op}, we have
\begin{eqnarray}\label{Cont_mean2_X}
\lim_{|b-a|\to 0} \Exp{ |X(t,b) - X(t,a)|^2 } = 0,
\end{eqnarray}
i.e., for each $t\in\R$, $X(t,.)$ is continuous in the mean-squared. In particular, for each $\epsilon>0$
$$ \lim_{|b-a|\to 0} \Prob{|X(t,b) - X(t,a)|>\epsilon} \leq \epsilon^{-2} \lim_{|b-a|\to 0} \Exp{ |X(t,b) - X(t,a)|^2 } =0. $$
Thus, $X(t,.)$ is continuous in probability. Next, for any $t\in\R$, we define $\mu_t(\{u\}) = 0$ and 
\begin{eqnarray}\label{def:mu}
\mu_t([u,v)) = \Exp{|X(t,u)-X(t,v)|^2}.
\end{eqnarray}
If $u<v<w$, by \eqref{US_X} we have
\begin{eqnarray*}
\mu_t([u,w)) &=& \Exp{|X(t,u) \pm X(t,v) -X(t,w)|^2} \\
&=& \Exp{|X(t,u)-X(t,v)|^2} + \Exp{|X(t,v)-X(t,w)|^2} \\ 
&& + 2\Exp{(X(t,u)-X(t,v))(X(t,v)-X(t,w))} \\
&=& \mu_t([u,v)) + \mu_t([v,w))
\end{eqnarray*}
so $\mu_t$ is an additive set function of intervals. Moreover,
\begin{lemma}\label{lemma:mu}
$\mu_t$ is a premeasure defined on the set algebra of the finite unions of the intervals. 
\end{lemma}
\begin{proof}
Clearly $\mu_t\geq 0$ and $\mu_t(\emptyset) = 0$. It is enough to prove that if $([a_k,b_k))_{k\in\N}$ are disjoint intervals such that $$[a,b) = \bigcup_{k\in\N} [a_k,b_k)$$ for some $a,b\in\R$, then $$\mu_t([a,b)) = \sum_{k\in\N} \mu_t([a_k,b_k)).$$ The other cases are similar, and the desired result follows from this at once.

Now, for any $n\in\N$, the complement of the finite union of disjoint intervals $([a_k,b_k))_{k=0}^n$ in $[a,b)$ is a finite union of disjoint intervals in the same form $([c_k,d_k))_{k=0}^{m_n}$. Thus, we have a disjoint union $$[a,b) = \bigcup_{k=0}^n [a_k,b_k) \cup \bigcup_{k=0}^{m_n} [c_k,d_k)$$ where clearly $|m_n-n|\leq 1$. Since $\mu_t$ is finitely additive, we have 
\begin{eqnarray}\label{eq:mu1}
\mu_t([a,b)) = \sum_{k=0}^n \mu_t([a_k,b_k)) + \sum_{k=0}^{m_n} \mu_t([c_k,d_k)) \geq \sum_{k=0}^n \mu_t([a_k,b_k)).
\end{eqnarray}
Increasing and bounded from above, the sequence of the partial sums in \eqref{eq:mu1} converges:
\begin{eqnarray}\label{eq:mu2}
\mu_t([a,b)) \geq \sum_{k\in\N} \mu_t([a_k,b_k))
\end{eqnarray}
Second, for any $\delta>0$, there is a $N\in\N$ such that $$\sum_{k=0}^{m_N} (d_k-c_k) <\delta.$$ By \eqref{Cont_mean2_X}, for $m\in\N$ and for any $\epsilon>0$, there is a $\delta>0$, $$\sum_{k=0}^{m} (d_k-c_k) <\delta \Longrightarrow \sum_{k=0}^{m} \mu_t([c_k,d_k)) <(m+1)\epsilon.$$ Puting these two together, for any $\epsilon>0$, there is a $N\in\N$ $$\mu_t([a,b)) - \sum_{k\in\N} \mu_t([a_k,b_k)) \leq \mu_t([a,b)) - \sum_{k=0}^N \mu_t([a_k,b_k)) = \sum_{k=0}^{m_N} \mu_t([c_k,d_k)) < (m_N+1)\epsilon.$$ Left hand side of the last inequality doesn't depend on $N$ or $\epsilon$. Thus, as $\epsilon\to 0$, we get the equality in \eqref{eq:mu2}
\end{proof}

A premeasure can be extended to a measure in a very standart way, e.g., see \cite{Folland84}. This measure is defined on the $\sigma$-algebra of sets generated by the intervals of the form $[a,b)$, which is the Borel algebra of $\R$. With no confusion, we call this measure also $\mu_t$. Since $\mu_t$ is a continuous $\sigma$-finite Borel measure, it enjoys similar properties as the Lebesgue measure. We summarize some of these properties which we shall need later in Lemma~\ref{lemma:L2_mu} without proof.

\begin{lemma}\label{lemma:L2_mu}
Let $p\in[1,\infty)$. For any $t\in\R$, $L^p(\R,\mu_t)$ is translation invariant, and step functions are dense in $L^p(\R,\mu_t)$. Moreover, any bounded Borel measurable function with compact support is in $L^p(\R,\mu_t)$, and the space $C_c^{\infty}(\R)$ of infinitely differentiable functions with compact support is dense in $L^p(\R,\mu_t)$.
\end{lemma}

Now we can define the stochastic integral with respect to $X$, in very much the same way the Ito integral is defined (e.g., see \cite{Evans12,Oksendal}). Let $t\in\R$ be fixed. First, let $$\phi = \sum_{k=1}^n c_k \1_{[u_{k-1},u_k)}$$ be a step function in $L^2(\R,\mu_t)$, and define $$I_t(\phi) = \sum_{k=1}^n c_k ( X(t,u_k) - X(t,u_{k-1}) ).$$ Then, by \eqref{US_X}, 
\begin{eqnarray*}
\Exp{|I_t(\phi)|^2} 
&=& \sum_{k,l=1}^n c_k \overline{c_l} \Exp{ ( X(t,u_k) - X(t,u_{k-1}) )( X(t,u_l) - X(t,u_{l-1}) )  } \\
&=& \sum_{k=1}^n |c_k|^2 \Exp{ |X(t,u_k) - X(t,u_{k-1})|^2 } \\
&=& \sum_{k=1}^n |c_k|^2 \mu_t([u_{k-1},u_k)) \\
&=& \int_{\R} |\phi|^2\ d\mu_t.
\end{eqnarray*}
In other words, if $(\Omega,P)$ is the underlying probability space, then $\|I_t(\phi)\|_{L^2(\Omega,P)} = \|\phi\|_{L^2(\R,\mu_t)}$, where $L^2(\Omega,P)$ is the set of all random variables $Z$ on the probability space, for which $\|Z\|_{L^2(\Omega,P)} = \Exp{|Z|^2} <\infty$. Thus, by Lemma~\ref{lemma:L2_mu}, $I_t:L^2(\R,\mu_t) \to L^2(\Omega,P)$ is a linear isometry, defined on a dense subset of $L^2(\R,\mu_t)$. Consequently, it has a unique extension to all of $L^2(\R,\mu_t)$.

\begin{notation}
For all $f\in L^2(\R,\mu_t)$, we use the customary notation 
\begin{eqnarray}\label{def:I_t(f)}
I_t(f) := \int f(u) X(t,du)
\end{eqnarray}
\end{notation}

\begin{remark}\normalfont
We defined $I_t(f)$ in \eqref{def:I_t(f)} for deterministic $f$. We could have defined it for a class of predictable random functions in almost the same way. 
\end{remark}

$I_t(f)$ that we have just defined is nothing but $Hf(t)$. In fact, if $f = \1_{[a,b)}$, then $Hf(t) = X(t,b) - X(t,a) = I_t(f)$ by definition. Since $H$ is a linear operator, $H$ is defined on step functions. If $f = \sum_{k=1}^n c_k \1_{[u_{k-1},u_k)}$ is a step function, then $$Hf(t) = \sum_{k=1}^n c_k ( X(t,u_k) - X(t,u_{k-1}) ) = I_t(f).$$ Moreover, $Hf(t)$ is uniquely defined for $f\in L^2(\R,\mu_t)$ by $Hf(t) = I_t(f)$. As an operator, $H$ is defined on $\bigcap_{t\in\R} L^2(\R,\mu_t)$, which is a locally convex space with (semi)norms given by $\mu_t$.

We summarize some of our results in the following theorems. We have just proven Theorem~\ref{thm:SIO} and Theorem~\ref{thm:SIO_2}.

\begin{theorem}\label{thm:SIO}
$H$ is a random linear operator, which is defined on step functions, and satisfies \eqref{US_op} and \eqref{Cont_mean2_op} if and only if, for each $t\in\R$, there exist a stochastic process $\{ X(t,u): u\in\R\}$, which is continuous in mean-squared and has uncorrelated increments such that
\begin{equation}\label{SIO_representation}
Hf(t) = \int f(u) X(t,du)
\end{equation}
for every step function $f$. 
\end{theorem}

\begin{theorem}\label{thm:SIO_2}
Let $H$ be a random linear operator as in \eqref{SIO_representation}, and $\mu_t$ be the measure as in Lemma~\ref{lemma:mu} and Lemma~\ref{lemma:L2_mu}. Then, $Hf(t)$ is well-defined for every $f\in L^2(\R,\mu_t)$ and satisfies $$\Exp{|Hf(t)|^2} = \int_{\R}|f(u)|^2 \mu_t(du).$$  
\end{theorem}

We call $X(t,u)$ the integrated kernel, and $X(t,du)$ the kernel symbol of the operator $H$ given in \eqref{SIO_representation}. When $H$ satisfies \eqref{US_op}, we have $$\Exp{\iint f(u)g(v) X(t,du) X(t,dv) } = \int f(u)g(u)\ \mu_t(du) $$ We use the short notation 
\begin{eqnarray}\label{eq:short_X}
\Exp{X(t,du) X(t,dv)} = \mu_t(du) \delta(u-v).
\end{eqnarray}

\subsection{Impulse response}\label{sec:SIO_rep_Y}

\begin{proposition}\label{prop:Y_int}
Let $Y(t,u) = X(t,t)-X(t,t-u)$. For each $t$ fixed, the process $\{Y(t,u):u\in\R\}$ satisfies the following.
\begin{enumerate}
\item $\Exp{Y(t,u)} = 0$.

\item $\{Y(t,u):u\in\R\}$ has uncorrelated increments and continuous in mean-squared. 

\item For any step function $\phi = \sum_{k=0}^{n-1} c_k\1_{[u_k,u_{k+1})}$
$$\Exp{|\sum_{k=0}^{n-1} c_k ( Y(t,u_{k+1}) - Y(t,u_k) ) |^2} = \int |\phi(t-u)|^2\ \mu_t(du) $$

\item For any $f\in L^2(\R,\mu_t)$, both integrals below exist and are equal a.s.
\begin{eqnarray}\label{eq:Y_int}
\int f(u) Y(t,du) = \int f(t-u) X(t,du).
\end{eqnarray}
\end{enumerate}
\end{proposition}
\begin{proof}
First and second parts are immediate by definition, since $X$ has zero mean, uncorrelated increments and continuous in mean-squared. 

Third, $$\phi(u) = \sum_{k=0}^{n-1} c_k\1_{[u_k,u_{k+1})}(u)
\Rightarrow \phi(t-u) = \sum_{k=0}^{n-1} c_{n-k-1} \1_{(\tilde{u}_{k},\tilde{u}_{k+1}]}(u).$$ Since $Y$ has independent increments,
\begin{eqnarray*}
&&\Exp{|\sum_{k=0}^{n-1} c_k ( Y(t,u_{k+1}) - Y(t,u_k) ) |^2} \\
&=& \sum_{k=0}^{n-1} \sum_{l=0}^{n-1} c_kc_l \Exp{( Y(t,u_{k+1}) - Y(t,u_k) )( Y(t,u_{l+1}) - Y(t,u_l) )} \\
&=& \sum_{k=0}^{n-1} |c_k|^2 \Exp{( X(t,t-u_{k+1}) - X(t,t-u_k) )^2} \\
&=& \sum_{k=0}^{n-1} |c_{n-k-1}|^2 \Exp{( X(t,\tilde{u}_{k+1}) - X(t,\tilde{u}_k) )^2} \\
&=& \sum_{k=0}^{n-1} |c_{n-k-1}|^2 \mu_t([\tilde{u}_k,\tilde{u}_{k+1})) \\
&=& \int |\phi(t-u)|^2 \ \mu_t(du)
\end{eqnarray*}
where $u_{n-k} = t - \tilde{u}_{k}$. Just as in the definition of the integral in \eqref{def:I_t(f)}, we have a linear isometry, which extends to $L^2(\R,\mu_t)$.

Fourth, for the step function $\phi$, we have a.s.
\begin{eqnarray*}
\int \phi(u)Y(t,du)
&=& \sum_{k=0}^{n-1} c_k ( Y(t,u_{k+1}) - Y(t,u_k) ) \\
&=& \sum_{k=0}^{n-1} c_{n-k-1} ( X(t,\tilde{u}_{k+1}) - X(t,\tilde{u}_k) ) \\
&=& \int \phi(t-u)X(t,du)
\end{eqnarray*}
Next, let $f\in L^2(\R,\mu_t)$. For every $\epsilon>0$, there exists a step function $\phi$ such that
$$\int|f(t-u) - \phi(t-u)|^2\ \mu_t(du) < \epsilon^2.$$ Then,
\begin{eqnarray*}
&&\Exp{ | \int f(t-u)X(t,du) - \int f(u)Y(t,du)  |^2}^{1/2} \\
&&\hspace{20mm} < \Exp{ |  \int f(u)Y(t,du) - \int \phi(u)Y(t,du)|^2 }^{1/2}  \\
&&\hspace{35mm} + \Exp{ |  \int \phi(t-u)X(t,du) - \int \phi(u)Y(t,du)|^2 }^{1/2}  \\
&&\hspace{35mm} + \Exp{ |  \int f(t-u)X(t,du) - \int \phi(t-u)X(t,du)|^2 }^{1/2} \\
&&\hspace{20mm} = 2 \left(\int|f(t-u) - \phi(t-u)|^2\ \mu_t(du)\right)^{1/2} \\
&&\hspace{20mm} < 2\epsilon.
\end{eqnarray*}
As $\epsilon\to 0$, we obtain the identity \eqref{eq:Y_int}.
\end{proof}

\subsection{Kohn-Nirenberg symbol}\label{sec:SIO_rep_sigma}

\begin{proposition}\label{prop:sigma}
For each $T>0$, let 
\begin{eqnarray}\label{def:sigma}
\sigma_T(t,\xi) = \int_{-T}^T e^{-2\pi iu\xi} Y(t,du)
\end{eqnarray}
Then, the process $\sigma_T$ satisfies the following.
\begin{enumerate}
\item $\Exp{\sigma_T(t,\xi)}=0.$
\item For $t,\xi_1,\xi_2\in\R$ and $T_1>T_2\geq T_3>T_4$, 
$$\Exp{(\sigma_{T_1}(t,\xi_1) - \sigma_{T_2}(t,\xi_1))
(\overline{\sigma_{T_3}(t,\xi_2)} - \overline{\sigma_{T_4}(t,\xi_2)}) } = 0.$$
\item For $t,\xi\in\R$ and $T>S$, 
$$\Exp{|\sigma_{T}(t,\xi) - \sigma_{S}(t,\xi)|^2 } = \mu_t(\{t+u:S<|u|<T\}).$$ In particular, $$\lim_{|S-T|\to 0} \Exp{|\sigma_{T}(t,\xi) - \sigma_{S}(t,\xi)|^2 } = 0.$$
\end{enumerate}
\end{proposition}
\begin{proof}
By \eqref{eq:Y_int} and \eqref{eq:short_X}, we get
\begin{eqnarray*}
&&\Exp{(\sigma_{T_1}(t,\xi_1) - \sigma_{T_2}(t,\xi_1))
(\overline{\sigma_{T_3}(t,\xi_2)} - \overline{\sigma_{T_4}(t,\xi_2)}) } \\
&=& \Exp{\iint \1_{[T_2,T_1)}(|u|) \1_{[T_4,T_3)}(|v|) e^{-2\pi i(v-u)\xi} Y(t,du)Y(t,dv)} \\
&=& \Exp{\iint \1_{[T_2,T_1)}(|t-u|) \1_{[T_4,T_3)}(|t-v|) e^{-2\pi i(v-u)\xi} X(t,du)X(t,dv)} \\
&=& \int \1_{[T_2,T_1)}(|t-u|) \1_{[T_4,T_3)}(|t-u|) \mu_t(du) \\
&=& 0.
\end{eqnarray*}
Similarly, 
\begin{eqnarray*}
\Exp{|\sigma_{T}(t,\xi) - \sigma_{S}(t,\xi)|^2 } 
= \int \1_{[S,T)}(|t-u|) \mu_t(du) 
= \mu_t(\{t+u:S<|u|<T\}).
\end{eqnarray*}
\end{proof}

\begin{lemma}\label{lemma:stochastic_Fubini}
For any $t\in\R$, $T>0$ and $f\in C_c^{\infty}(\R)$, we have a.s.
\begin{eqnarray}\label{eq:stochastic_Fubini}
\int_{-T}^T f(t-u) Y(t,du) = \int e^{2\pi it\xi} \widehat{f}(\xi) \sigma_T(t,\xi) d\xi.
\end{eqnarray}
\end{lemma}
\begin{proof}
\eqref{eq:stochastic_Fubini} is the same as
\begin{eqnarray*}
\int_{t-T}^{t+T} \int e^{2\pi iu\xi} \widehat{f}(\xi) d\xi X(t,du) = \int \int_{t-T}^{t+T} e^{2\pi iu\xi} \widehat{f}(\xi) X(t,du) d\xi,
\end{eqnarray*}
which is a stochastic Fubini theorem. First, notice that for each $\epsilon>0$, there are step functions $g_k,h_k$, $k=1\dots,n$ such that $$|e^{2\pi iu\xi}\widehat{f}(\xi) - \sum_{k=1}^n g_k(u)h_k(\xi)| \leq \frac{\epsilon}{1 + \xi^2}$$ for all $u\in[t-T,t+T]$ and all $\xi\in\R$. Thus, 
\begin{eqnarray*}
&&\Exp{\left|\int_{t-T}^{t+T} \int e^{2\pi iu\xi} \widehat{f}(\xi) - \sum_{k=1}^n g_k(u)h_k(\xi)\hspace{3mm}d\xi X(t,du) \right|^2} \\
&& \hspace{25mm} = \int_{t-T}^{t+T} \left|\int e^{2\pi iu\xi} \widehat{f}(\xi) - \sum_{k=1}^n g_k(u)h_k(\xi)\hspace{3mm}d\xi\right|^2 \mu_t(du) \\
&&\hspace{25mm} \leq \int_{t-T}^{t+T} \left(\int \frac{\epsilon}{1 + \xi^2} d\xi\right)^2 \mu_t(du) \\
&&\hspace{25mm} \leq \epsilon^2\pi^2\mu_t([t-T,t+T))
\end{eqnarray*}

Also, 
\begin{eqnarray*}
&& \Exp{\left|\int \int_{t-T}^{t+T} e^{2\pi iu\xi} \widehat{f}(\xi) - \sum_{k=1}^n g_k(u)h_k(\xi)\hspace{3mm}X(t,du) d\xi \right|^2}^{1/2} \\
&&\hspace{25mm} \leq \int \Exp{\left|\int_{t-T}^{t+T} e^{2\pi iu\xi} \widehat{f}(\xi) - \sum_{k=1}^n g_k(u)h_k(\xi)\hspace{3mm}X(t,du) \right|^2}^{1/2} d\xi \\
&&\hspace{25mm} = \int \left(\int_{t-T}^{t+T} |e^{2\pi iu\xi} \widehat{f}(\xi) -\sum_{k=1}^n g_k(u)h_k(\xi)|^2 \mu_t(du)\right)^{1/2} d\xi \\
&&\hspace{25mm} \leq \int \sqrt{\mu_t([t-T,t+T)) } \frac{\epsilon}{1 + \xi^2}  d\xi \\
&&\hspace{25mm} = \epsilon\pi\sqrt{\mu_t([t-T,t+T)) }
\end{eqnarray*}

Together, we get
\begin{eqnarray*}
&&\Exp{\left|\int_{t-T}^{t+T} \int e^{2\pi iu\xi} \widehat{f}(\xi) d\xi X(t,du) 
- \int \int_{t-T}^{t+T} e^{2\pi iu\xi} \widehat{f}(\xi) X(t,du) d\xi\right|^2}^{1/2} \\
&&\hspace{15mm} \leq \Exp{\left|\int_{t-T}^{t+T} \int e^{2\pi iu\xi} \widehat{f}(\xi) - \sum_{k=1}^n g_k(u)h_k(\xi)\hspace{3mm}d\xi X(t,du) \right|^2}^{1/2} \\
&&\hspace{25mm} + \Exp{\left|\int \int_{t-T}^{t+T} e^{2\pi iu\xi} \widehat{f}(\xi) - \sum_{k=1}^n g_k(u)h_k(\xi)\hspace{3mm}X(t,du) d\xi \right|^2}^{1/2} \\
&&\hspace{15mm} \leq 2\epsilon\pi\sqrt{\mu_t([t-T,t+T)) }
\end{eqnarray*}
As $\epsilon\to 0$, we get the result.
\end{proof}

\vspace{3mm}

For a fixed $t\in\R$, \eqref{eq:stochastic_Fubini} gives the linear functional
\begin{eqnarray}\label{eq:sigma_int_0}
I_t(f) = \lim_{T\to\infty} \int e^{2\pi it\xi} \widehat{f}(\xi) \sigma_T(t,\xi) d\xi
\end{eqnarray} 
for all $f\in C_c^{\infty}(\R)$. We had defined $I_t(f)$ in \eqref{def:I_t(f)} and shown that $\Exp{|I_t(f)|^2} = \|f\|_{L^2(\R,\mu_t)}^2$ for every $f\in L^2(\R,\mu_t)$. On the other hand, $C_c^{\infty}(\R)$ is dense in $L^2(\R,\mu_t)$ by Lemma~\ref{lemma:L2_mu}. Thus, the functional in the right hand side of \eqref{eq:sigma_int_0} has a unique extension to $L^2(\R,\mu_t)$ and must be equal to $I_t(f)$.

We use the simpler notation $$\int e^{2\pi it\xi} \widehat{f}(\xi) \sigma(t,\xi) d\xi$$ when we want to emphasize the right hand size of \eqref{eq:sigma_int_0}, even though the integral in \eqref{eq:sigma_int_0} might not formally exist for all $f\in L^2(\R,\mu_t)$, or $(\sigma_T)_{T>0}$ might not converge weakly to a {\it function} $\sigma$ as $T\to\infty$. In this case, $\sigma$ shall just be a  symbol for this linear functional, which is the same as the {\it Kohn-Nirenberg symbol} of $H$.

With this notation in hand, we state the following theorem, which we just proved.

\begin{theorem}\label{thm:sigma_int}
For any $f\in L^2(\R,\mu_t)$, we have a.s.
\begin{eqnarray}\label{eq:sigma_int}
Hf(t) = \int e^{2\pi it\xi} \widehat{f}(\xi) \sigma(t,\xi) d\xi.
\end{eqnarray}
\end{theorem}

\subsection{Spreading symbol}\label{sec:SIO_rep_eta}

If for each $u\in\R$ fixed, the sample paths of $\{Y(t,u):t\in\R\}$ are locally integrable a.s., then we define 
\begin{eqnarray}\label{def:eta}
\eta_T(u,\gamma) = \int_{-T}^T e^{-2\pi it\gamma} Y(t,u) dt
\end{eqnarray}
for any $T>0$.

\begin{theorem}\label{thm:eta_int}
For any $t\in\R$, $T>0$ and $f\in L^2(\R,\mu_t)$, we have a.s.
\begin{eqnarray}\label{eq:eta_int}
\1_{[-T,T]}(t) \int f(t-u) Y(t,du) = \iint e^{2\pi it\gamma} f(t-u) \eta_T(du,\gamma) d\gamma.
\end{eqnarray}
\end{theorem}
\begin{proof}
First, let $f=\sum_{k=0}^{n-1} c_k\1_{[u_k,u_{k+1})}$ be a step function. For each $t\in\R$,
\begin{eqnarray*}
\iint e^{2\pi it\gamma} f(t-u) \eta_T(du,\gamma) d\gamma
&=& \int e^{2\pi it\gamma} \sum_{k=0}^{n-1} c_k ( \eta_T(t-u_{k},\gamma)-\eta_T(t-u_{k+1},\gamma) ) d\gamma \\
&=& \int e^{2\pi it\gamma} \sum_{k=0}^{n-1} c_k \int_{-T}^{T} e^{-2\pi ix\gamma}  ( Y(x,t-u_{k})-Y(x,t-u_{k+1}) ) dx d\gamma \\
&=& \1_{[-T,T]}(t) \sum_{k=0}^{n-1} c_k  ( Y(t,t-u_{k})-Y(t,t-u_{k+1}) ) \\
&=& \1_{[-T,T]}(t) \int f(t-u) Y(t,du) 
\end{eqnarray*}
Thus, for a fixed $t\in\R$ and $T>0$, \eqref{eq:eta_int} gives the linear functional
\begin{eqnarray}\label{eq:eta_T_int}
I^T_t(f) = \1_{[-T,T]}(t)I_t(f) = \iint e^{2\pi it\gamma} f(t-u) \eta_T(du,\gamma) d\gamma
\end{eqnarray} 
for all step functions $f$. We had proven that $I_t$ has a unique extension to $L^2(\R,\mu_t)$, which satisfies $\Exp{|I_t(f)|^2} = \|f\|_{L^2(\R,\mu_t)}^2$. Consequently,  $I^T_t$ has a unique extension to $L^2(\R,\mu_t)$, which satisfies $\Exp{|I^T_t(f)|^2} = \1_{[-T,T]}(t)\|f\|_{L^2(\R,\mu_t)}^2$. Thus, we obtained \eqref{eq:eta_int}.
\end{proof}

\vspace{3mm}

We proceed similar to the previous subsection. For a fixed $t\in\R$, \eqref{eq:eta_int} gives the linear functional
\begin{eqnarray}\label{eq:eta_int_0}
I_t(f) = \lim_{T\to\infty} \iint e^{2\pi it\gamma} f(t-u) \eta_T(du,\gamma) d\gamma
\end{eqnarray} 
for all $f\in L^2(\R,\mu_t)$. In order to ease the notation, we write 
\begin{eqnarray}\label{eq:eta_int_1}
I_t(f) = \iint e^{2\pi it\gamma} f(t-u) \eta(du,\gamma) d\gamma
\end{eqnarray}
when we want to emphasize the right hand size of \eqref{eq:eta_int_0}, even though $(\eta_T)_{T>0}$ might not converge weakly to a {\it function} $\eta$ as $T\to\infty$. In this case, $\eta$ shall just be a symbol for the limiting linear functional $I_t$ and thus for the SIO $H$. $\eta$ is the {\it spreading symbol} for the SIO $H$.

\section{WSSUS property}\label{sec:WSSUS}

We have not considered what the US and WSSUS properties translate into for SIOs. If $H$ satisfies US assumption, then \eqref{def:US} and Proposition~\ref{prop:Y_int} together imply
$$
\int P(t,s,u) f(t-u) g(s-u)\ du
= \Exp{Hf(t) Hg(s)}
= \Exp{\iint f(t-u)g(s-v) Y(t,du) Y(s,dv) }
$$
In particular, when $f(u) = \1_{[a,b)}(t-u)$, $g(u) = \1_{[a,b)}(s-u)$, we have 
\begin{eqnarray}\label{eq:rho_US}
\int_{[a,b)\cap[c,d)} P(t,s,u) \ du = \Exp{ (Y(t,b)-Y(t,a))(Y(s,d)-Y(s,c)) }
\end{eqnarray}
If $H$ satisfies WSSUS assumption, we have the same result with $P(t,s,u)$ replaced by $P(t-s,u)$. 
Accordingly, we start with the definition of the {\it correlation measure} $\rho_{s,t}$.

\subsection{Correlation measure}

For any $t,s\in\R$, we define $\rho_{s,t}(\{a\}) = 0$ and
\begin{eqnarray}\label{def:rho}
\rho_{s,t}([a,b)) = \Exp{ (Y(t,b)-Y(t,a))(Y(s,b)-Y(s,a)) }
\end{eqnarray}
It is readily seen in \eqref{eq:rho_US} that, not only are the increments of $Y(t,.)$ uncorrelated, but also the increments of $Y(t,.)$ and $Y(s,.)$ are all uncorrelated, i.e., for $s,t\in\R$ and $u_1>u_2\geq u_3>u_4$
\begin{eqnarray}\label{def:rho_US}
\Exp{ (Y(t,u_1)-Y(t,u_2))(Y(s,u_3)-Y(s,u_4)) } = 0.
\end{eqnarray}
Notice that \eqref{def:rho} and \eqref{def:rho_US} together imply
\begin{eqnarray*}
\rho_{s,t}([a,b)\cap [c,d)) = \Exp{(Y(t,b)-Y(t,a))(Y(s,d)-Y(s,c))}
\end{eqnarray*}
Also, notice that $\rho_{t,t}([a,b)) = \mu_t((t-b,t-a])$.

If $u<v<w$, by \eqref{def:rho} and \eqref{def:rho_US} we have
\begin{eqnarray*}
\rho_{s,t}([u,w)) &=& \Exp{(Y(t,w) \pm Y(t,v)- Y(t,u))(Y(s,w) \pm Y(s,v) -Y(s,u))} \\
&=& \Exp{(Y(t,w) - Y(t,v))(Y(s,w) - Y(s,v)) } \\ 
&& + \Exp{(Y(t,w) - Y(t,v))(Y(s,v) -Y(s,u))} \\
&& + \Exp{(Y(t,v)- Y(t,u))(Y(s,w) - Y(s,v))} \\
&& + \Exp{(Y(t,v)- Y(t,u))(Y(s,v) -Y(s,u))} \\
&=& \rho_{s,t}([u,v)) + \rho_{s,t}([v,w)).
\end{eqnarray*}
Thus, $\rho_{s,t}$ is uniquely defined as an additive set function on the set algebra of the finite unions of intervals.
Similar to Lemma~\ref{lemma:mu}, we have the following.
\begin{lemma}\label{lem:rho}
For every $s,t\in\R$ and $T>0$, $\rho_{s,t}$ can be uniquely extended to a signed finite Borel measure on $[-T,T]$. Furthermore, for every Borel set $B\subseteq [-T,T]$, 
\begin{eqnarray}\label{eq:rho_corr_var}
|\rho_{s,t}(B)|^2 \leq \rho_{s,s}(B)\rho_{t,t}(B).
\end{eqnarray}
\end{lemma}
\begin{proof}
For any interval $I$, $\rho_{t,t}(I) = \mu_t(t-I)$. This uniquely defines $\rho_{t,t}$ as a $\sigma$-finite (positive) Borel measure on $\R$ since $\mu_t$ is a $\sigma$-finite (positive) Borel measure on $\R$. 

Second, for any $s,t\in\R$ fixed, consider the process $\{Y(t,u) + Y(s,u) : u\in\R\}$. Since $X$ satisfies \eqref{US_X} and \eqref{Cont_mean2_X}, so does $\{Y(t,u) + Y(s,u) : u\in\R\}$. Thus, we
define $\beta_{s,t}(\{a\})=0$ and
$$\beta_{s,t}([a,b)) = \Exp{|(Y(t,b) + Y(s,b)) - (Y(t,a) + Y(s,a))|^2}.$$  
One can proceed as in Lemma~\ref{lemma:mu} and show that each $\beta_{s,t}$ can be extended to a $\sigma$-finite Borel measure on $\R$. Particularly, if $\{(a_k,b_k):k\in J\}$ is a countable set of disjoint intervals and \\ $O = \cup_{k\in J}(a_k,b_k)$ is their union, then by \eqref{def:rho_US} we have 
\begin{eqnarray*}
\beta_{s,t}(O)  
=\sum_{k\in J} \beta_{s,t}((a_k,b_k))
=\Exp{|\sum_{k\in J}(Y(t,b_k) + Y(s,b_k)) - (Y(t,a_k) + Y(s,a_k))|^2}
\end{eqnarray*}
Similarly, by \eqref{def:rho_US}
\begin{eqnarray*}
\rho_{t,t}(O)
&=&\Exp{|\sum_{k\in J}Y(t,b_k) - Y(t,a_k)|^2} \\
\rho_{s,t}(O)  
&=&\Exp{(\sum_{k\in J}Y(t,b_k) - Y(t,a_k))(\sum_{l\in J}Y(s,b_l) - Y(s,a_l))}
\end{eqnarray*}
Using the last three equations, we easily obtain
\begin{eqnarray}
\beta_{s,t}(O) &=& \rho_{t,t}(O) + 2\rho_{s,t}(O) + \rho_{s,s}(O)  \label{eq:rho_meas_O1} \\
\sqrt{\beta_{s,t}(O)} &\leq& \sqrt{\rho_{t,t}(O)} + \sqrt{\rho_{s,s}(O)} \label{eq:rho_meas_O2}
\end{eqnarray}
For any Borel set $B$, there are countable unions of intervals $(O_n)_{n\in\N}$ such that \\ 
$$\lim_{n\to\infty}\beta_{s,t}(O_n) = \beta_{s,t}(B),\hspace{4mm} 
\lim_{n\to\infty}\rho_{t,t}(O_n) = \rho_{t,t}(B),\hspace{4mm} 
\lim_{n\to\infty}\rho_{s,s}(O_n) = \rho_{s,s}(B).$$
Write \eqref{eq:rho_meas_O1} and \eqref{eq:rho_meas_O2} with $O_n$ and let $n\to\infty$. On any bounded interval $[-T,T]$ we obtain 
\begin{eqnarray}
2\rho_{s,t}(B) &=& \beta_{s,t}(B) - \rho_{t,t}(B) - \rho_{s,s}(B)  \label{eq:rho_meas_B1} \\
\sqrt{\beta_{s,t}(B)} &\leq& \sqrt{\rho_{t,t}(B)} + \sqrt{\rho_{s,s}(B)} \label{eq:rho_meas_B2}
\end{eqnarray}
Note that on any bounded interval, since $\beta_{s,t}, \rho_{t,t}, \rho_{s,s}$ are finite measures, $\lim_{n\to\infty}\rho_{s,t}(O_n)$ exists.
Now, \eqref{eq:rho_meas_B1} defined $\rho_{s,t}$ uniquely as a signed finite Borel measure on $[-T,T]$.

Third, it is immediate from \eqref{eq:rho_meas_B1} and \eqref{eq:rho_meas_B2} that $$\rho_{s,t}(B) \leq \sqrt{\rho_{t,t}(B)\rho_{s,s}(B)} $$ If we started with the process $\{Y(t,u) - Y(s,u) : u\in\R\}$, we would similarly obtain $$- \rho_{s,t}(B) \leq \sqrt{\rho_{t,t}(B)\rho_{s,s}(B)} .$$ Hence, we obtained \eqref{eq:rho_corr_var}.
\end{proof}
The integral with respect to $\rho_{s,t}$ is well defined for functions $f\in L^2(\R,\rho_{t,t})\cap L^2(\R,\rho_{s,s})$ with bounded support.

Now, for fixed $s,t\in\R$ and $f,g$ step functions, we define $J_{s,t}(f,g) = \Exp{Hf(t)Hg(s)}.$ Since $f,g$ are step functions, we can write 
\begin{eqnarray}\label{eq:J_st(fg)_step_functions}
f(t-u) = \sum_{k=0}^{n-1} b_k\1_{[u_{k},u_{k+1})}(u)\hspace{9mm} g(s-u) = \sum_{k=0}^{n-1} c_k\1_{[u_{k},u_{k+1})}(u) 
\end{eqnarray}
for some $u_0<u_1<\dots<u_n$. Then, by \eqref{def:rho} and \eqref{def:rho_US},
\begin{eqnarray}
J_{s,t}(f,g)
&=& \sum_{k,l=0}^{n-1} b_k c_l \Exp{ ( Y(t,u_{k+1}) - Y(t,u_{k}) )( Y(s,u_{l+1}) - Y(s,u_{l}) )  } \nonumber\\
&=& \sum_{k=0}^{n-1} b_kc_k \Exp{ ( Y(t,u_{k+1}) - Y(t,u_{k}) )( Y(s,u_{k+1}) - Y(s,u_{k}) )  } \nonumber\\
&=& \sum_{k=0}^{n-1} b_kc_k \hspace{2mm} \rho_{s,t}([u_{k},u_{k+1})) \nonumber\\
&=& \int_{\R} f(t-u)g(s-u)\ \rho_{s,t}(du). \label{eq:J_st(fg)}
\end{eqnarray}
Also, 
\begin{eqnarray}
|J_{s,t}(f,g)|
&\leq& \sum_{k=0}^{n-1} |b_k||c_k| \hspace{2mm} |\rho_{s,t}(  [u_{k},u_{k+1})  )| \nonumber\\
&\leq& \sum_{k=0}^{n-1} |b_k||c_k| \hspace{2mm} \sqrt{\rho_{s,s}([u_{k},u_{k+1}))\rho_{t,t}([u_{k},u_{k+1}))}\nonumber\\
&\leq& \sqrt{
\sum_{k=0}^{n-1} |b_k|^2 \hspace{2mm} \rho_{t,t}( [u_{k},u_{k+1}) ) 
\sum_{k=0}^{n-1} |c_k|^2 \hspace{2mm} \rho_{s,s}( [u_{k},u_{k+1}) ) } \nonumber\\
&=& \|f\|_{L^2(\R,\mu_t)} \|g\|_{L^2(\R,\mu_s)} \label{eq:J_st(fg)_cont}
\end{eqnarray}
Thus, $J_{s,t}$ is a bounded bilinear functional defined on a dense subspace of $L^2(\R,\mu_t)\times L^2(\R,\mu_s)$. Hence, $J_{s,t}$ has a unique continuous extension to $L^2(\R,\mu_t)\times L^2(\R,\mu_s)$. We consider this extension when we write \eqref{eq:J_st(fg)} and \eqref{eq:J_st(fg)_cont} for all $f\in L^2(\R,\mu_t)$ and $g\in L^2(\R,\mu_s)$, especially when $u\to f(t-u)g(s-u)$ is not integrable with respect to the signed measure $\rho_{s,t}$.

Rewriting \eqref{eq:J_st(fg)}, for all $f\in L^2(\R,\mu_t)$ and $g\in L^2(\R,\mu_s)$ we have $$\Exp{\iint f(t-u)g(s-u) Y(t,du)Y(s,dv)} = \int_{\R} f(t-u)g(s-u)\ \rho_{s,t}(du).$$ We use the short notation 
\begin{eqnarray}\label{eq:short_WSSUS_Y}
\Exp{Y(t,du)Y(s,dv)} = \rho_{s,t}(du) \delta(u-v)
\end{eqnarray}
In this line, the WSSUS assumption equivalently translates into $$\rho_{s,t} = \rho_{0,t-s}$$ for every $t,s\in\R$. We write $\rho_{t-s}$ instead of $\rho_{0,t-s}$ for simpler the notation.

In summary, IS, US and WSSUS assumptions translate into Definition~\ref{def:WSSUS_SIO} for SIOs.
\begin{definition}\label{def:WSSUS_SIO}\normalfont
Let $H$ be a SIO with integrated impulse response $Y$. $H$ satisfies the US assumption if for every $s,t\in\R$, $[a,b)\cap[c,d)=\emptyset$ implies 
$$\Exp{ (Y(t,b)-Y(t,a))(Y(s,d)-Y(s,c)) } = 0.$$
$H$ satisfies the IS assumption if $(Y(t,b)-Y(t,a))$ and $(Y(s,d)-Y(s,c))$ are not only uncorrelated but also independent random variables. 
In addition, if $\rho_{s,t}=\rho_{t-s}$, 
then $H$ satisfies the WSSUS (WSSIS, resp.) assumption.
\end{definition}

\begin{theorem}
If $H$ is a SIO that satisfies WSSUS assumption, then it is a linear operator defined on $L^2(\R,\mu)$. For every $f\in L^2(\R,\mu)$, we have 
\begin{eqnarray}\label{eq:WSSUS_H_isom}
\Exp{|Hf(t)|^2} = \int_{\R} |f(u+t)|^2 \mu(du). 
\end{eqnarray}
In particular, if $\nu$ is another Borel measure, then 
\begin{eqnarray}\label{eq:WSSUS_H_isom_2}
\Exp{\|Hf\|^2_{L^2(\R,\nu)}} = \|f\|^2_{L^2(\R,\mu*\nu)}.
\end{eqnarray}
\end{theorem}
\begin{proof}
Remember that $\mu=\mu_0$. For each $t\in\R$, $L^2(\R,\mu_t)$ is translation invariant by Lemma~\ref{lemma:L2_mu}. By the WSSUS assumption, we also have
$$\mu_t([a,b)) = \rho_{t,t}((t-b,t-a]) = \rho_{0,0}((t-b,t-a]) = \mu_0([a,b)-t).$$
Thus, $L^2(\R,\mu_t) = L^2(\R,\mu)$, and for every $f\in L^2(\R,\mu)$ 
$$\int_{\R} |f(u)|^2 \mu_t(du) = \int_{\R} |f(u+t)|^2 \mu(du)$$ for all $t\in\R$. Hence we obtain \eqref{eq:WSSUS_H_isom}. Finally, if $\nu$ is another Borel measure, then $$ \Exp{ \int |Hf(t)|^2 \nu(dt) } = \iint |f(u+t)|^2 \mu(du)\nu(dt) = \int |f(y)|^2 \mu*\nu(dy) $$ which is \eqref{eq:WSSUS_H_isom_2}.
\end{proof}

\subsection{Scattering measure}

For any bounded Borel set $B\subseteq\R$ and $\gamma,\widetilde{\gamma}\in\R$, let
\begin{eqnarray}\label{def:scattering_0}
\scattering_{S,T}(\gamma,\widetilde{\gamma},B) = \int_{-T}^T\int_{-S}^S e^{2\pi i(s\widetilde{\gamma}-t\gamma)} \rho_{s,t}(B)  dsdt
\end{eqnarray}

\begin{proposition}\label{prop:eta_US}
If $H$ is a SIO with the US property, then $\eta_T$ defined in \eqref{def:eta} satisfies the following.
\begin{enumerate}
\item It has uncorrelated increments, i.e., for $u_1>u_2\geq u_3>u_4$, $\gamma,\widetilde{\gamma}\in\R$ and $T,S>0$
\begin{eqnarray*}
\Exp{(\eta_T(u_1,\gamma) - \eta_T(u_2,\gamma))(\overline{\eta_S(u_3,\widetilde{\gamma})} - \overline{\eta_S(u_4,\widetilde{\gamma}))}} =0
\end{eqnarray*}
\item It is continuous in the mean squared, i.e., for any $\gamma\in\R$ and $T>0$
\begin{eqnarray*}
\lim_{|b-a|\to 0} \Exp{|\eta_T(b,\gamma) - \eta_T(a,\gamma)|^2} = 0.
\end{eqnarray*}
\end{enumerate}
\end{proposition}
\begin{proof}
First, observe that
\begin{eqnarray*}
&& \Exp{(\eta_T(u_1,\gamma) - \eta_T(u_2,\gamma))(\overline{\eta_S(u_3,\widetilde{\gamma})} - \overline{\eta_S(u_4,\widetilde{\gamma}))}} \\
&=& \int_{-T}^T\int_{-S}^S e^{2\pi i(s\widetilde{\gamma}-t\gamma)} \Exp{(Y(t,u_1)-Y(t,u_2))(Y(s,u_3)-Y(s,u_4))} dtds \\
&=& \int_{-T}^T\int_{-S}^S e^{2\pi i(s\widetilde{\gamma}-t\gamma)} \rho_{s,t}([u_4,u_3)\cap[u_2,u_1)) dtds \\
&=& \scattering_{S,T}(\gamma,\widetilde{\gamma},[u_4,u_3)\cap[u_2,u_1))
\end{eqnarray*}
If $[u_4,u_3)$ and $[u_2,u_1)$ are disjoint, we get the uncorrelated increments property. 

Second, 
\begin{eqnarray*}
\Exp{|\eta_T(b,\gamma) - \eta_T(a,\gamma)|^2}
&=& \scattering_{T,T}(\gamma,\gamma,[a,b)) \\
&\leq& \int_{-T}^T\int_{-T}^T |\rho_{s,t}([a,b))| \ dsdt \\
&\leq& \left(\int_{-T}^T \sqrt{\rho_{t,t}([a,b))} \ dt\right)^2
\end{eqnarray*}
The last integral goes to $0$ as $|b-a|\to 0$ by the Lebesgue dominated convergence theorem.
\end{proof}

\vspace{3mm}

For every $s,t\in\R$, $f\in L^2(\R,\mu_t)$ and $g\in L^2(\R,\mu_s)$, \eqref{eq:J_st(fg)} and \eqref{def:scattering_0} together imply that
\begin{eqnarray}\label{eq:scattering_ST}
\1_{[-T,T]}(t)\1_{[-S,S]}(s) J_{s,t}(f,g) 
=\iiint e^{2\pi i(t\gamma-s\widetilde{\gamma})} f(t-u) g(s-u) \scattering_{S,T}(\gamma,\widetilde{\gamma},du)  d\gamma d\widetilde{\gamma}
\end{eqnarray}
Clearly, $\1_{[-T,T]}(t)\1_{[-S,S]}(s) J_{s,t}$ converges to $J_{s,t}$ as $T,S\to \infty$. In order to simplify the notation, we write
\begin{eqnarray}\label{eq:scattering}
J_{s,t}(f,g) 
=\iiint e^{2\pi i(t\gamma-s\widetilde{\gamma})} f(t-u) g(s-u) \scattering(\gamma,\widetilde{\gamma},du)  d\gamma d\widetilde{\gamma}
\end{eqnarray}
even though $(\scattering_{S,T})$ might not converge weakly to a function or a measure as $T,S\to \infty$. In this case, $\scattering$ will only be a symbol for $J_{s,t}$. $\scattering$ is the {\it scattering measure} of the SIO $H$.

In order to obtain the counterpart of \eqref{eq:short_WSSUS_Y} for $\scattering$ , we remind 
\begin{eqnarray*}
\1_{[-T,T]}(t)\1_{[-S,S]}(s) J_{s,t}(f,g)
= \Exp{\iiiint e^{2\pi i(t\gamma-s\widetilde{\gamma})} f(t-u) g(s-v) \eta_T(du,\gamma) \overline{\eta_S(dv,\widetilde{\gamma})} d\gamma d\widetilde{\gamma}}
\end{eqnarray*}
and 
\begin{eqnarray*}
J_{s,t}(f,g)
= \Exp{\iiiint e^{2\pi i(t\gamma-s\widetilde{\gamma})} f(t-u) g(s-v) \eta(du,\gamma) \overline{\eta(dv,\widetilde{\gamma})} d\gamma d\widetilde{\gamma}}
\end{eqnarray*}
Considering these with \eqref{eq:scattering_ST} and \eqref{eq:scattering}, we derive the short notation
\begin{eqnarray}
\Exp{\eta_T(du,\gamma)\overline{\eta_S(dv,\widetilde{\gamma})}} &=& \scattering_{S,T}(\gamma,\widetilde{\gamma},du)\delta(u-v) \label{eq:short_eta_T} \\
\Exp{\eta(du,\gamma)\overline{\eta(dv,\widetilde{\gamma})}} &=& \scattering(\gamma,\widetilde{\gamma},du)\delta(u-v) \label{eq:short_eta}
\end{eqnarray}
WSSUS assumption is satisfied if and only if $\rho_{s,t} = \rho_{t-s}$ if and only if 
\begin{eqnarray}
\Exp{\eta_T(du,\gamma)\overline{\eta_S(dv,\widetilde{\gamma})}} &=& \scattering_{S,T}(\gamma,du)\delta(\gamma-\widetilde{\gamma})\delta(u-v) \label{eq:short_WSSUS_eta_T}\\
\Exp{\eta(du,\gamma)\overline{\eta(dv,\widetilde{\gamma})}} &=& \scattering(\gamma,du)\delta(\gamma-\widetilde{\gamma})\delta(u-v)  \label{eq:short_WSSUS_eta}
\end{eqnarray}

\section{Independent decomposition of SIO}\label{sec:H_decomp}

For each $t\in\R$, $X(t,.)$ is continuous in probability by \eqref{Cont_mean2_X}. If $H$ satisfies the weak-IS property, then $X(t,.)$ has independent increments by \eqref{US_X}. Then by Proposition~\ref{prop:Y_int}, $Y(t,.)$ is continuous in probability and has independent increments, as well. Thus, both $X(t,.)$ and $Y(t,.)$ becomes {\it additive process}es. In this section, our purpose is to convey some of the properties of the additive processes into our setting. Therefore, we often omit the proofs of the classical results, and refer the reader to classical references for the proofs.

A stochastic process $(Z_u)$ is customarily defined for $u\geq 0$. The additive processes are defined for  $u\geq 0$ in Section~\ref{sec:additive_process}. However, $X(t,u)$ and the related processes are defined for $u\in\R$. 
Thus, the definitions and the theorems in Section~\ref{sec:additive_process} might look slightly different when we apply them in Section~\ref{sec:decomposition}.

\subsection{Overview of Additive Processes}\label{sec:additive_process}

A stochastic process $(Z_u)$, which is continuous in probability and has independent increments, is called an {\it additive process} \cite{Sato99}. It is customary to require $Z_0=0$ a.s. In addition, if $(Z_u)$ has stationary increments, it is a {\it Levy process}. 

Although stationary increments property is missing, a version of Levy-Khinchine Theorem holds for additive processes. 

\begin{theorem}\label{thm:Levy-Khinchine}
Let $(Z_u)$ be an additive process with probability distributions $(\zeta_u)$. Then, each $\zeta_u$ is infinitely divisible probability measure. Furthermore, for $a>0$, there exist
\begin{enumerate}
\item continuous $m:[0,\infty)\to\R$  with $m(0) = 0$,
\item continuous increasing $\alpha:[0,\infty)\to\R$ with $\alpha(0)=0$,
\item $\nu$ (Levy) measure on $\R\times\R\backslash\{0\}$ satisfying $\nu(\{u\}\times B)= 0$, $\nu_u(\{0\}) = 0$, $\nu_0(B) = 0$ and $$\int_{\R} \min(1,x^2) \nu_u(dx) <\infty$$ for every $u\in\R$ and every Borel set $B\subseteq\R\backslash\{0\}$, where $$\nu([w,u)\times B) = \nu_u(B) - \nu_w(B) $$

\end{enumerate}
such that the characteristic function $\widehat{\zeta_u}$ is given by 
\begin{eqnarray}\label{eq:Levy-Khinchine}
\widehat{\zeta_u}(\gamma) = \exp\left(-i\gamma m(u) - \alpha(u)\frac{\gamma^2}{2} 
+ \int_{|y|\geq a}(e^{iy\gamma} -1)\ \nu_u(dy) 
+ \int_{|y|<a}(e^{iy\gamma} -1 -iy\gamma)\ \nu_u(dy)  \right)  
\end{eqnarray}
Furthermore, if $Z$ also has stationary increments, i.e., if it is a Levy process, then we get $m(u) = um(1)$, $\alpha(u)=u\alpha(1)$, $\nu_u(B) = u\nu_{1}(B)$ in \eqref{eq:Levy-Khinchine}. 
\end{theorem}
\begin{proof}
See \cite{Kallenberg97} Chapter~13, \cite{Sato99} Section~9.
\end{proof}

Just as Theorem~\ref{thm:Levy-Khinchine} is a version of Levy-Khinchine Theorem for aditive processes, there is a version of Levy-Ito decomposition for additive processes. We start with some definitions.

\begin{definition}\label{def:Poisson-random-measure}
Let $(Z_u)$ be a rcll (right continuous left limit) stochastic process. 
\begin{enumerate}
\item Let $\Delta Z_u = Z_u - Z_{u-}$ where $Z_{u-} = \lim_{w\to u-} Z_w$. $(\Delta Z_u)$ is the {\it jump process} of $(Z_u)$.

\item For any Borel set $B\subseteq\R\backslash\{0\}$, let $J^u_B = \{ 0<w\leq u: \Delta Z_w \in B \}$, and
\begin{eqnarray}\label{def:N}
N(u,B) = |J^u_B| = \sum_{0<w\leq u} \1_B(\Delta Z_w).
\end{eqnarray}
Note that $J^u_B$ is the set of jump points, where the jump size is equal to some $x\in B$. Thus, for each $u\in\R$ fixed, $N$ is a random counting measure.

\item Let $\lambda_u(B) = \Exp{N(u,B)}$ and $\widetilde{N}(u,B) = N(u,B) - \lambda_u(B)$. $\widetilde{N}$ is the compensated random measure associated to $N$.
\end{enumerate}
\end{definition}

Note that, by definition, for each $u\in\R$ $$N(u,B) = \sum_{x\in B} N(u,\{x\}),$$ and same is true for $\widetilde{N}$. Thus, for a nonnegative Borel measurable function $g:\R\to\R$, 
\begin{eqnarray}\label{def:poisson_integral_1}
\int_B g(y)N(u,dy) 
= \sum_{x\in B} g(x)N(u,\{x\}) 
= \sum_{w\in J^u_B} g(\Delta Z_w)
\end{eqnarray}
In fact, this holds for all Borel functions by Theorem~\ref{thm:Poisson_random_measure}.1. Second, \eqref{def:poisson_integral_1} defines a random process with piecewise constant sample paths, which has jumps exactly at $u\in J_B = \{w\geq 0: \Delta Z_w\in B\}$ with jump size $g(\Delta Z_u)$. Thus, for any rcll function $f:\R\to\R$, $T>0$
\begin{eqnarray}\label{def:poisson_integral_2}
\int_{0}^T f(u) \int_B g(y)N(du,dy) 
= \sum_{u\in J^T_B} f(u-) g(\Delta Z_u)
\end{eqnarray}

If $Z$ is an additive process, then there is a rcll process equal to $Z$ a.s. (see, e.g. \cite{Applebaum09, Kallenberg97,Sato99}). Thus, $N$ and $\widetilde{N}$ can be defined for an additive process as in Definition~\ref{def:Poisson-random-measure}. Furthermore,
\begin{theorem}\label{thm:Poisson_random_measure}
Let $B$ be a Borel subset of $\R$, $0\notin\overline{B}$. Then, 
\begin{enumerate}
\item $J^u_B$ is a finite set a.s., thus $N(u,B)<\infty$ for a.e. $u>0$.

\item $\{N(u,B):u>0\}$ is an inhomogeneous (unless $\lambda_u(B) = u\lambda_1(B)$) Poisson process with 
$$\Prob{N(u,B) = k} = e^{-\lambda_u(B)} \frac{(\lambda_u(B))^k}{k!}$$
for $k\in\N$, where $\lambda_u(B) = \Exp{N(u,B)}$.

\item If $(B_k)_{k=1}^n$ are disjoint Borel subsets of $\R\backslash\{0\}$, then for each $u$, $(N(u,B_k))_{k=1}^n$ are independent random variables.

\item If $(B_k)_{k=1}^n$ are disjoint Borel subsets of $\R\backslash\{0\}$, $(f_k)_{k=1}^n$ are Borel measurable functions, let $$Y^k_u = \int_{B_k} f_k(y) N(u,dy)$$ then $Y^1,\dots,Y^n$ are independent processes.

\item For a Borel measurable function $f:\R\to\R$, 
\begin{eqnarray*}
\Exp{\exp\left(i\gamma\int_B f(y)N(u,dy)\right)} 
= \exp\left( \int_B (e^{i\gamma f(y)}-1)\ \lambda_u(dy) \right)
\end{eqnarray*}
and 
\begin{eqnarray*}
\Exp{\exp\left(i\gamma\int_B f(y)\widetilde{N}(u,dy)\right)} 
= \exp\left( \int_B (e^{i\gamma f(y)}-1-i\gamma f(y))\ \lambda_u(dy) \right)
\end{eqnarray*}

\end{enumerate}
\end{theorem}
\begin{proof}
See \cite{Sato99} Chapter~4, or \cite{Applebaum09} Section~2.3.
\end{proof}

Let $J^u_a = \{0<w\leq u: |\Delta Z_w|\geq a\}$ and let 
$$P^a_u = \int_{|y|\geq a} y N(u,dy) = \sum_{w\in J^u_a} \Delta Z_w $$
The process $P^a$ clearly has rcll piecewise constant sample paths. The jump points of $P^a$ are same as the jump points of $Z$ with jump size greater than $a$. Then, the sample paths of $Z-P^a$ are rcll, and the jump sizes at points of discontinuity are less than $a$. 

$(P^a)_{a>0}$ generally does not converge in probability. 
Therefore, $N$ is replaced by $\widetilde{N}$ in order to capture smaller jumps. Let $$\widetilde{P}^{\epsilon,a}_u = \int_{\epsilon<|y|<a} y \widetilde{N}(u,dy)$$ Then, $(\widetilde{P}^{\epsilon,a})_{\epsilon>0}$ converges in probability to a stochastic process $\widetilde{P}^a$ (see \cite{Applebaum09,Sato99}). Notationally, it is simply written $$\widetilde{P}^a_u = \int_{|y|<a} y \widetilde{N}(u,dy).$$ $P^a$ and $\widetilde{P}^a$ are independent processes as a result of Theorem~\ref{thm:Poisson_random_measure}.4.

Since $P^a + \widetilde{P}^a$ captures all of the jumps of $Z$, the sample paths of $Z-P^a-\widetilde{P}^a$ are continuous. The characteristic function of $P^a + \widetilde{P}^a$ is 
$$\Exp{e^{i\gamma (P^a_u + \widetilde{P}^a_u)}} 
= \exp\left( \int_{|y|\geq a} (e^{i\gamma y}-1)\ \lambda_u(dy) 
+ \int_{|y|<a} (e^{i\gamma y}-1-i\gamma y)\ \lambda_u(dy)  \right) .$$ 
$\lambda_u$ really is equal to $\nu_u$ in Theorem~\ref{thm:Levy-Khinchine} \cite{Sato99}. 

Below is the version of Levy-Ito decomposition for additive processes.
\begin{theorem}\label{thm:Levy-Ito}
Let $Z$ be an additive process. For $a>0$, there exist
\begin{enumerate}
\item continuous $m:[0,\infty)\to\R$  with $m(0) = 0$,
\item continuous increasing $\alpha:[0,\infty)\to\R$ with $\alpha(0)=0$, and a Gaussian process $G$ with $$\Exp{e^{i\gamma G_u}} = e^{- \alpha(u)\gamma^2/2}$$
\item a Poisson random measure $N$ given by \eqref{def:N}, 
\end{enumerate}
such that
\begin{eqnarray}\label{eq:Levy-Ito}
Z_u + m(u) = G_{u} + \int_{|y|\geq a} y N(u,dy) + \int_{|y|<a} y \widetilde{N}(u,dy)
\end{eqnarray}
The three random processes that appear in the right hand side of \eqref{eq:Levy-Ito} are independent additive processes.
\end{theorem}
\begin{proof}
See \cite{Kallenberg97} Chapter~13, \cite{Sato99} Chapter~4.
\end{proof}

Let $B$ be a standart Brownian motion. Then, $(G_u)$ and $(B_{\alpha(u)})$ has the same probability distribution, as well as their increments. In fact, let $0<v<u$. Since the increments of $G$ are independent,
\begin{eqnarray*}
e^{- \alpha(u)\gamma^2/2} 
= \Exp{e^{i\gamma (G_u-G_0)}} 
= \Exp{e^{i\gamma (G_u-G_v)}}\Exp{e^{i\gamma (G_v-G_0)}}
= \Exp{e^{i\gamma (G_u-G_v)}} e^{- \alpha(v)\gamma^2/2}
\end{eqnarray*}
Thus, $$\Exp{e^{i\gamma (G_u-G_v)}} = e^{- (\alpha(u)-\alpha(v))\gamma^2/2} = \Exp{e^{i\gamma (B_{\alpha(u)}-B_{\alpha(v)})}}$$ where the second equality is by the definition of the Brownian motion. Consequently, increments of $(G_u)$ and $(B_{\alpha(u)})$ has the same probability distribution. Since $G_0 = B_0 = 0$, these two processes have te same probability distribution as well.

Second, if $\alpha$ is absolutely continuous, we define $$Y_u = \int_0^u \sqrt{\alpha'_w}\ dB_w$$ Let let $0<v<u$. Then, $$\Exp{e^{i\gamma (Y_u-Y_v)}} = e^{- (\alpha(u)-\alpha(v))\gamma^2/2}$$ i.e., increments of $(Y_u)$ has the same probability distribution as $(G_u)$ and $(B_{\alpha(u)})$. Since $Y_0 = 0$, all three processes have the same distribution as well.

\begin{theorem}\label{thm:Independent_additive_processes}
Let $a>0$ and let $Z,W$ be two additive processes with Levy-Ito decompositions
\begin{eqnarray*}
Z_u + m(u) &=& G_{u} + P^a_u + \widetilde{P}^a_u \\
W_u + \mu(u) &=& \Gamma_{u} + \Pi^a_u + \widetilde{\Pi}^a_u 
\end{eqnarray*}
$Z,W$ are independent if and only if $G, \Gamma, P^a, \Pi^a, \widetilde{P}^a, \widetilde{\Pi}^a$ are pairwise independent processes.
\end{theorem}
\begin{proof}
If $Z,W$ are independent, then $\{\Delta Z,W\}$, $\{Z, \Delta W\}$ and $\{\Delta Z, \Delta W\}$ are sets of independent processes. Consequently, $\{P^a, \widetilde{P}^a, W\}$, $\{Z, \Pi^a, \widetilde{\Pi}^a\}$, $\{P^a, \widetilde{P}^a, \Pi^a, \widetilde{\Pi}^a\}$, and so $\{Z, \Gamma\}$, $\{W, G\}$, $\{P^a, \widetilde{P}^a, \Gamma\}$, $\{G, \Pi^a, \widetilde{\Pi}^a\}$, $\{G,\Gamma\}$ are sets of pairwise independent additive processes. 

The other direction is obvious.
\end{proof}

\subsection{The decomposition of weak-IS channels}\label{sec:decomposition}

For each $t$ fixed, $\{X(t,u):u\in\R\}$ is an additive process. $X$ is associated with the 
jump process $\Delta X(t,u) = X(t,u) - X(t,u-)$, 
the random jump measure $$N(t,u,B) = |\{0<\frac{w}{u}\leq 1: \Delta X(t,w) \in B \}| = \sum_{0<\frac{w}{u}\leq 1} \1_B(\Delta X(t,w)) ,$$ for $u\neq 0$ and $N(t,0,B)=0$,
the compensated jumps random measure $$\widetilde{N}(t,u,B) = N(t,u,B) - \Exp{N(t,u,B)},$$ and
the Levy measure $$\nu(t,[w,u)\times B) = \left\{
\begin{array}{ll}
\Exp{N(t,u,B) - N(t,w,B)} &: 0\leq w<u \\
\Exp{N(t,u,B) + N(t,w,B)} &: w<0\leq u \\
\Exp{N(t,w,B) - N(t,u,B)} &: w<u\leq 0 \\
\end{array}\right.$$
We also use the notation $\nu_u(t,B)= \Exp{N(t,u,B)}$. 

By Theorem~\ref{thm:Levy-Ito}, $X$ has a Levy-Ito decomposition 
\begin{eqnarray}\label{eq:decomposition_X}
X(t,u) = - X_d(t,u) + X_c(t,u) + X_j(t,u) + \widetilde{X}_j(t,u)
\end{eqnarray}
where, for each $t\in\R$ fixed, 
\begin{enumerate}
\item $X_d$ is a nonrandom continuous function of $u$ with $X_d(t,0)=0$.
\item $X_c$ is a Gaussian process with the characteristic function 
$$\Exp{e^{i\gamma X_c(t,u)}} = \left\{
\begin{array}{ll} 
e^{- \alpha_t(u)\gamma^2/2} &; u\geq 0 \\
e^{ \alpha_t(u)\gamma^2/2} &; u<0 \\
\end{array}\right.$$ 
where for each $t\in\R$, $\alpha_t:\R\to\R$ is a continuous increasing function with $\alpha_t(0)=0$.
\item $X_j$ and $\widetilde{X}_j$ are explicitely given by
\begin{eqnarray*}
X_j(t,u) = \int_{|y|\geq a} y N(t,u,dy) \hspace{15mm}
\widetilde{X}_j(t,u) = \int_{|y|<a} y \widetilde{N}(t,u,dy)
\end{eqnarray*}
with the characteristic functions
\begin{eqnarray*}
\Exp{e^{i\gamma X_j(t,u)}} &=& \exp\left(\int_{|y|\geq a} (e^{iy\gamma}-1) \nu_u(t,dy) \right) \\ 
\Exp{e^{i\gamma \widetilde{X}_j(t,u)}} &=& \exp\left(\int_{|y|<a} (e^{iy\gamma}-1-iy\gamma) \nu_u(t,dy) \right)
\end{eqnarray*}
respectively. We shall not explicitely indicate the dependence on $a>0$, since it will always be a fixed number in this paper.  
\end{enumerate}
In particular, $\Exp{X_c(t,u)}=\Exp{\widetilde{X}_j(t,u)}=0$. Since $\Exp{X(t,u)}=0$, we have 
\begin{eqnarray}\label{eq:Exp_Xj=Xd}
X_d(t,u) = \Exp{X_j(t,u)} = \int_{|y|\geq a} y\nu_u(t,dy).
\end{eqnarray}
Naturally, the stochastic integral operator $H$ has the decomposition 
\begin{eqnarray}\label{eq:decomposition_H}
H = -H_d + H_c + H_j + \widetilde{H}_j
\end{eqnarray}
By Theorem~\ref{thm:SIO}, each of the four operators in \eqref{eq:decomposition_H} are SIO, which satisfy \eqref{US_op} and \eqref{Cont_mean2_op}, with the integrated kernels given in the same order in \eqref{eq:decomposition_X}.
In particular, note that
\begin{eqnarray*}
H_df(t) &=& \int f(u)\int_{|y|\geq a} y\nu(t,dudy)\\
H_jf(t) &=& \sum_{u\in J_a} f(u) \Delta X(t,u)
\end{eqnarray*}
for a step function $f$, where $J_a = \{u\in\R: |\Delta X(t,u)|\geq a\}$. $H_df(t) = \Exp{H_jf(t)}$ by \eqref{eq:Exp_Xj=Xd}.

\begin{lemma}\label{lem:Characteristic_functions_X_increments}
For each $t\in\R$, the increments of $X_c$, $X_j$ and $\widetilde{X}_j$ have the following characteristic functions.
\begin{eqnarray*}
\Exp{e^{i\gamma (X_c(t,v)-X_c(t,u))}} &=& e^{-(\alpha_t(v) - \alpha_t(u))\gamma^2/2} \\
\Exp{e^{i\gamma (X_j(t,v) - X_j(t,u))}} &=& \exp\left(\int_u^v\int_{|y|\geq a} (e^{iy\gamma}-1) \nu(t,dwdy) \right) \\
\Exp{e^{i\gamma (\widetilde{X}_j(t,v) - \widetilde{X}_j(t,u))}} &=& \exp\left(\int_u^v\int_{|y|<a} (e^{iy\gamma}-1-iy\gamma) \nu(t,dwdy) \right)
\end{eqnarray*}
\end{lemma}
\begin{proof}
We shall consider only the case $0\leq u<v$, since the cases $u< 0\leq v$ and $u<v\leq 0$ can be handled similarly.

For each $t$, $X_c(t,.)$ has independent increments. Consequently, 
\begin{eqnarray*}
e^{-\alpha_t(v)\gamma^2/2} 
&=& \Exp{e^{i\gamma (X_c(t,v)-X_c(t,0))}} \\
&=& \Exp{e^{i\gamma (X_c(t,v)-X_c(t,u))}} \Exp{e^{i\gamma (X_c(t,u)-X_c(t,0))}} \\
&=& \Exp{e^{i\gamma (X_c(t,v)-X_c(t,u))}} e^{-\alpha_t(u)\gamma^2/2}.
\end{eqnarray*}
Remember that $X_c(t,0)=0$ a.s.

Similarly, $X_j(t,.)$ and $\widetilde{X}_j(t,.)$ have independent increments. Therefore, the characteristic functions of their increments are
\begin{eqnarray*}
\exp\left(\int_{|y|\geq a} (e^{iy\gamma}-1) \nu_v(t,dy)\right)
&=& \Exp{e^{i\gamma (X_j(t,v) - X_j(t,0))}} \\
&=& \Exp{e^{i\gamma (X_j(t,v) - X_j(t,u))}}\Exp{e^{i\gamma (X_j(t,u) - X_j(t,0))}} \\
&=& \Exp{e^{i\gamma (X_j(t,v) - X_j(t,u))}}\exp\left(\int_{|y|\geq a} (e^{iy\gamma}-1) \nu_u(t,dy)\right)
\end{eqnarray*}
Remeber that $X_j(t,0)=0$ a.s., and $\nu(t,[u,v)\times B) = \nu_v(t,B) - \nu_u(t,B)$. 

The third characteristic function is obtained similarly.
\end{proof}

Next, similar to the definition of $\mu_t$ in \eqref{def:mu}, we define 
\begin{eqnarray*}
\mu^c_t([u,v)) &=& \Exp{|X_c(t,u)-X_c(t,v)|^2} \\
\mu^j_t([u,v)) &=& \Exp{|X_j(t,u) - X_d(t,u) + X_d(t,v) - X_j(t,v)|^2} \\
\widetilde{\mu}^j_t ([u,v)) &=& \Exp{|\widetilde{X}_j(t,u)-\widetilde{X}_j(t,v)|^2} 
\end{eqnarray*}
and proceed as in Section~\ref{sec:SIO_rep_X}. $\mu^c_t$, $\mu^j_t$ and $\widetilde{\mu}^j_t$ are $\sigma$-finite Borel measures that satisfy Lemma~\ref{lemma:L2_mu}. Then, by Theorem~\ref{thm:SIO_2}
$H_cf(t)$ is well-defined for every $f\in L^2(\R,\mu^c_t)$ and satisfies 
$$\Exp{|H_cf(t)|^2} = \|f\|_{L^2(\R,\mu^c_t)}^2,$$ 
$H_jf(t)$ is well-defined for every $f\in L^2(\R,\mu^j_t)$ and satisfies 
$$\Exp{|H_jf(t) - H_df(t)|^2} = \|f\|_{L^2(\R,\mu^j_t)}^2,$$
$\widetilde{H}_jf(t)$ is well-defined for every $f\in L^2(\R,\widetilde{\mu}^j_t)$ and satisfies 
$$\Exp{|\widetilde{H}_jf(t)|^2} = \|f\|_{L^2(\R,\widetilde{\mu}^j_t)}^2.$$
Since for each $t\in\R$ fixed, $X_c$, $X_j - X_d$, $\widetilde{X}_j$ are independent additive processes, $$\mu_t = \mu^c_t + \mu^j_t + \widetilde{\mu}^j_t.$$ As a result, $L^2(\R,\mu_t) = L^2(\R,\mu^c_t) \cap L^2(\R,\mu^j_t) \cap L^2(\R,\widetilde{\mu}^j_t)$. 

\begin{lemma}\label{lemma:mu_decomp}
Let $t\in\R$ be fixed. For any Borel set $B\subseteq\R$, 
\begin{eqnarray*}
\mu^c_t(B) &=& \int_B \alpha_t(du) \\
\mu^j_t(B) &=& \int_B \int_{|y|\geq a} y^2\nu(t,dudy) \\
\widetilde{\mu}^j_t (B) &=& \int_B \int_{|y|<a} y^2\nu(t,dudy) \\
\end{eqnarray*}
Furthermore, if for each $t$, $X(t,.)$ 
is a Levy process, then 
\begin{eqnarray*}
\mu^c_t(B) &=& |B|\alpha_t(1) \\
\mu^j_t(B) &=& |B| \int_{|y|\geq a} y^2\nu_1(t,dy) \\
\widetilde{\mu}^j_t (B) &=& |B| \int_{|y|<a} y^2\nu_1(t,dy) \\
\end{eqnarray*}
 where $|B|$ is the Lebesgue measure of $B$.
\end{lemma}
\begin{proof}
By Lemma~\ref{lem:Characteristic_functions_X_increments}, the variance of the increment $X_c(t,u)-X_c(t,v)$ is $$\mu^c_t([u,v)) = \Exp{|X_c(t,u)-X_c(t,v)|^2} = \alpha_t(v) - \alpha_t(u) = \int_{[u,v)} \alpha_t(dw).$$ Since both the left and the right hand side of above equality are Borel measures, and since they are equal on the intervals, they must be equal for all Borel sets.

Similarly, the variance of the increment $X_j(t,u)-X_j(t,v)$ is 
$$
\mu^j_t([u,v)) 
= \Exp{|X_j(t,u)-X_j(t,v)|^2} - |X_d(t,u)-X_d(t,v)|^2 
= \int_u^v \int_{|y|\geq a} y^2\nu(t,dwdy).
$$ 
The rest follows similarly. We obtain $\widetilde{\mu}^j_t (B)$ similarly as well.

Second, if $X(t,.)$ is a Levy process, then $\alpha_t(u) = u\alpha_t(1)$ and $$\nu(t,[u,v)\times A) = |v-u|\nu(t,[0,1)\times A) = |v-u|\nu_1(t,A)$$ for any Borel set $A\subseteq\R\backslash\{0\}$. The condition for the Levy measure could be written more generally as $\nu(t,B\times A) = |B|\nu_1(t,A)$. The result follows from here.
\end{proof}

\begin{theorem}\label{thm:Characteristic_functions}
For each $t\in\R$ and $f\in L^2(\R,\mu_t)$, $H_df(t)$ is deterministic; $H_cf(t)$, $H_jf(t)$ and $\widetilde{H}_jf(t)$ are independent random variables with the following characteristic functions
\begin{eqnarray*}
\Exp{e^{i\gamma H_cf(t)}} &=& \exp{\left(-\frac{\gamma^2}{2} \int f^2(u)\alpha_t(du)\right)} \\
\Exp{e^{i\gamma H_jf(t)}} &=& \exp{\left(\iint_{|y|\geq a} (e^{i\gamma yf(u)} - 1) \nu(t,dudy)\right)} \\
\Exp{e^{i\gamma \widetilde{H}_jf(t)}} &=& \exp{\left(\int\int_{|y|<a} (e^{i\gamma yf(u)} - 1-i\gamma yf(u)) \nu(t,dudy)\right)}
\end{eqnarray*}
\end{theorem}
\begin{proof}
We prove the theorem in two steps. We first prove it for step functions, then use the density argument by Lemma~\ref{lemma:L2_mu} to obtain the general result. 

Let $u_0<u_1<\dots<u_n$ and let $\phi = \sum_{k=1}^n c_k \1_{[u_{k-1},u_k)}$ be a step function. By Lemma~\ref{lem:Characteristic_functions_X_increments} and by the independence of the increments of $X_c$,
\begin{eqnarray*}
\Exp{e^{i\gamma H_c\phi(t)}} 
&=& \prod_{k=1}^n \Exp{e^{i\gamma c_k (X_c(t,u_k)-X_c(t,u_{k-1}))}} \\
&=& \prod_{k=1}^n e^{-(\alpha_t(u_k) - \alpha_t(u_{k-1}))c_k^2\gamma^2/2} \\
&=& \exp{\left(-\frac{\gamma^2}{2} \int \phi^2(u)\alpha_t(du)\right)}.
\end{eqnarray*}
Similarly, 
\begin{eqnarray*}
\Exp{e^{i\gamma H_j\phi(t)}} 
&=& \prod_{k=1}^n \Exp{e^{i\gamma c_k (X_j(t,u_k)-X_j(t,u_{k-1}))}} \\
&=& \prod_{k=1}^n \exp\left(\int_{u_{k-1}}^{u_k} \int_{|y|\geq a} (e^{iy\gamma c_k}-1) \nu(t,dudy) \right) \\
&=& \exp{\left(\iint_{|y|\geq a} (e^{i\gamma y\phi(u)} - 1) \nu(t,dudy)\right)}
\end{eqnarray*}
and
\begin{eqnarray*}
\Exp{e^{i\gamma \widetilde{H}_j\phi(t)}} 
&=& \prod_{k=1}^n \Exp{e^{i\gamma c_k (\widetilde{X}_j(t,u_k) - \widetilde{X}_j(t,u_{k-1}))}} \\
&=& \prod_{k=1}^n \exp\left(\int_{u_{k-1}}^{u_k} \int_{|y|\geq a} (e^{iy\gamma c_k}-1 - iy\gamma c_k) \nu(t,dudy) \right) \\
&=& \exp{\left(\iint_{|y|\geq a} (e^{i\gamma y\phi(u)} - 1 - iy\gamma \phi(u)) \nu(t,dudy)\right)}
\end{eqnarray*}

Second, by Lemma~\ref{lemma:L2_mu}, for every $f\in L^2(\R,\mu_t) = L^2(\R,\mu^c_t) \cap L^2(\R,\mu^j_t) \cap L^2(\R,\widetilde{\mu}^j_t)$, there is a sequence $(\phi_n)_{n\in\N}$ of step function such that 
\begin{eqnarray*}
\Exp{|H_cf(t)-H_c\phi_n(t)|^2} &=& \int |f(u) - \phi_n(u)|^2 \mu^c_t(du) \\
\Exp{|H_jf(t)-H_j\phi_n(t)|^2} - |H_df(t)-H_d\phi_n(t)|^2 &=& \int |f(u) - \phi_n(u)|^2 \mu^j_t(du) \\
\Exp{|\widetilde{H}_jf(t)-\widetilde{H}_j\phi_n(t)|^2} &=& \int |f(u) - \phi_n(u)|^2 \widetilde{\mu}^j_t(du)
\end{eqnarray*}
and these quantities tend to zero as $n\to\infty$. For the pair $(H_c,\mu_t^c)$,
\begin{eqnarray}
\left| \Exp{e^{i\gamma H_cf(t)}} - \Exp{e^{i\gamma H_c\phi_n(t)}} \right|
&\leq& \Exp{\left|e^{i\gamma H_cf(t)} - e^{i\gamma H_c\phi_n(t)} \right| } \nonumber \\
&\leq& |\gamma| \Exp{|H_cf(t) - H_c\phi_n(t)| } \nonumber \\
&\leq& |\gamma| \Exp{|H_cf(t) - H_c\phi_n(t)|^2 }^{1/2} \label{eq:Hc1}
\end{eqnarray}
On the other hand, by Lemma~\ref{lemma:mu_decomp} $\mu^c_t(du) = \alpha_t(du)$. Consequently,
\begin{eqnarray}
&& \lim_{n\to\infty} \int |f(u) - \phi_n(u)|^2 \alpha_t(du) = 0 \nonumber \\
&&\hspace{5mm} \Rightarrow \lim_{n\to\infty}\int \phi_n^2(u) \alpha_t(du) = \int f^2(u) \alpha_t(du) \nonumber \\
&&\hspace{5mm} \Rightarrow \lim_{n\to\infty} \exp{\left(-\frac{\gamma^2}{2} \int \phi_n^2(u)\alpha_t(du)\right)} = \exp{\left(-\frac{\gamma^2}{2} \int f^2(u)\alpha_t(du)\right)} \label{eq:Hc2}
\end{eqnarray}
By \eqref{eq:Hc1} and \eqref{eq:Hc2}
\begin{eqnarray*}
\Exp{e^{i\gamma H_cf(t)}} 
&=& \lim_{n\to\infty} \Exp{e^{i\gamma H_c\phi_n(t)}} \\
&=& \lim_{n\to\infty} \exp{\left(-\frac{\gamma^2}{2} \int \phi_n^2(u)\alpha_t(du)\right)} \\
&=& \exp{\left(-\frac{\gamma^2}{2} \int f^2(u)\alpha_t(du)\right)}
\end{eqnarray*}

For the pair $(H_j,\mu_t^j)$, we have 
\begin{eqnarray}
&& \lim_{n\to\infty} \left| \Exp{e^{i\gamma H_jf(t)}}e^{-i\gamma H_df(t)} - \Exp{e^{i\gamma H_j\phi_n(t)}} e^{-i\gamma H_d\phi_n(t)} \right| \nonumber \\
&=& \lim_{n\to\infty} \left| \Exp{e^{i\gamma (H_jf(t)-H_df(t))}} - \Exp{e^{i\gamma (H_j\phi_n(t)-H_d\phi_n(t))}} \right| \nonumber \\
&\leq& |\gamma| \lim_{n\to\infty} \Exp{|(H_jf(t)-H_df(t)) - (H_j\phi_n(t)-H_d\phi_n(t))|^2 }^{1/2} \nonumber \\
&=& 0 \label{eq:Hj1}.
\end{eqnarray}
On the other hand,
\begin{eqnarray*}
&& \left| \iint_{|y|\geq a} (e^{i\gamma yf(u)}-1-i\gamma yf(u))\nu(t,dudy)
 - \iint_{|y|\geq a} (e^{i\gamma y\phi_n(u)}-1-i\gamma y\phi_n(u))\nu(t,dudy) \right| \\
&&\hspace{19mm} \leq \iint_{|y|\geq a} \gamma^2|f(u)-\phi_n(u)|^2y^2 \nu(t,dudy) ,
\end{eqnarray*}
and by Lemma~\ref{lemma:mu_decomp},
\begin{eqnarray*}
\int |f(u) - \phi_n(u)|^2 \mu^j_t(du) = \iint_{|y|\geq a} |f(u) - \phi_n(u)|^2y^2\nu(t,dudy).
\end{eqnarray*}
Therefore,
\begin{eqnarray}\label{eq:Hj2}
\lim_{n\to\infty} \iint_{|y|\geq a} (e^{i\gamma y\phi_n(u)}-1-i\gamma y\phi_n(u))\nu(t,dudy) 
=\iint_{|y|\geq a} (e^{i\gamma yf(u)}-1-i\gamma yf(u))\nu(t,dudy)
\end{eqnarray}
By \eqref{eq:Hj1} and \eqref{eq:Hj2}
\begin{eqnarray*}
\Exp{e^{i\gamma H_jf(t)}} 
&=& \lim_{n\to\infty} \Exp{e^{i\gamma H_j\phi_n(t)}} e^{i\gamma H_d(f(t)-\phi_n(t))}\\
&=& \lim_{n\to\infty} \exp{\left(\iint_{|y|\geq a} (e^{i\gamma y\phi_n(u)} - 1) \nu(t,dudy)\right)} \exp{\left(i\gamma \iint_{|y|\geq a}(f(u)-\phi_n(u))y \nu(t,dudy)\right)} \\
&=& \exp{\left(\iint_{|y|\geq a} (e^{i\gamma yf(u)} - 1) \nu(t,dudy)\right)}
\end{eqnarray*}

Similar calculation gives the result for the pair $(\widetilde{H}_j,\widetilde{\mu}_t^j)$.
\end{proof}

For each $t$ fixed, $\{Y(t,u):u\in\R\}$ is an additive process. Thus, by Theorem~\ref{thm:Levy-Ito} and \eqref{eq:decomposition_X}, it has the Levy-Ito decomposition 
\begin{eqnarray}\label{eq:decomposition_Y}
Y(t,u) = - Y_d(t,u) + Y_c(t,u) + Y_j(t,u) + \widetilde{Y}_j(t,u)
\end{eqnarray}
where the processes in \eqref{eq:decomposition_Y} are the integrated impulse-responses for the SIO $H_d$, $H_c$, $H_j$ and $\widetilde{H}_j$ in the same order. 

Similar to the definition of $\mu_t$ in \eqref{def:rho}, we define 
\begin{eqnarray*}
\rho^c_{s,t}([u,v)) &=& \Exp{ (Y_c(t,v)-Y_c(t,u))(Y_c(s,v)-Y_c(s,u)) } \\
\rho^j_{s,t}([u,v)) &=& \Exp{ (Y_j(t,v)-Y_j(t,u))(Y_j(s,v)-Y_j(s,u)) } \\&&\hspace{10mm}
- (Y_d(t,v)-Y_d(t,u))(Y_d(s,v)-Y_d(s,u)) \\
\widetilde{\rho}^j_{s,t}([u,v)) &=& \Exp{ (\widetilde{Y}_j(t,v)-\widetilde{Y}_j(t,u))(\widetilde{Y}_j(s,v)-\widetilde{Y}_j(s,u))  }
\end{eqnarray*}
$\rho^c_{s,t}, \rho^j_{s,t}$ and $\widetilde{\rho}^j_{s,t}$ satisfy Lemma~\ref{lem:rho}.
\begin{theorem}\label{thm:IS_decomposition}
If $H$ satisfies the IS property, then $H_c, H_j$ and $\widetilde{H}_j$ also satisfy the IS property. Furthermore, if $\rho^c_{s,t}, \rho^j_{s,t}, \widetilde{\rho}^j_{s,t}$ are the correlation measures of $H_c, H_j, \widetilde{H}_j$ respectively, then 
\begin{eqnarray}\label{eq:IS_decomposition_rho}
\rho_{s,t} = \rho^c_{s,t} + \rho^j_{s,t} + \widetilde{\rho}^j_{s,t}
\end{eqnarray}
\end{theorem}
\begin{proof}
By Definition~\ref{def:WSSUS_SIO}, $H$ satisfies the IS assumption if and only if the increments \\ $(Y(t,u_1) - Y(t,u_2))$ and $(Y(s,u_3) - Y(s,u_4))$ are independent for $u_1>u_2\geq u_3>u_4$ and $s,t\in\R$. By Theorem~\ref{thm:Independent_additive_processes} 
\begin{eqnarray*}
Y_c(t,u_1) - Y_c(t,u_2),
Y_j(t,u_1) - Y_j(t,u_2),
\widetilde{Y}_j(t,u_1) - \widetilde{Y}_j(t,u_2), \\
Y_c(s,u_3) - Y_c(s,u_4),
Y_j(s,u_3) - Y_j(s,u_4),
\widetilde{Y}_j(s,u_3) - \widetilde{Y}_j(s,u_4),
\end{eqnarray*}
are independent processes. In particular, $H_c, H_j$ and $\widetilde{H}_j$ satisfy the IS property. 
Second, by independence, \eqref{eq:IS_decomposition_rho} is satisfied for a bounded interval. Since each $\rho_{s,t}, \rho^c_{s,t}, \rho^j_{s,t}, \widetilde{\rho}^j_{s,t}$ are signed Borel measures on any $[-T,T]$, then \eqref{eq:IS_decomposition_rho} is satisfied for all bounded Borel sets.
\end{proof}

As a consequence of Theorem~\ref{thm:IS_decomposition}, we have
\begin{eqnarray}\label{eq:IS_decomposition_scattering}
\scattering = \scattering^c + \scattering^j + \widetilde{\scattering}^j
\end{eqnarray}
where $\scattering^c, \scattering^j, \widetilde{\scattering}^j$ are the scattering measures of $H_c, H_j, \widetilde{H}_j$ respectively.

\section{Outlook}

In this article, we showed that US linear time variant channels can be modelled as stochastic integral operators, and vice versa. Particularly, IS linear time variant channels can be decomposed as a sum of four independent channels: a deterministic, a random Gaussian, and two random jump components. We explained the nature of the random processes behind these independent channels, and also provided some of their properties. However, numerous natural questions are left open to be answered.

We have not provided a way to separate the given independent channels $H_c$, $H_j$ and $\widetilde{H}_j$. Wavelets are known to be effective to detect the jumps of functions like the sample paths of the additive process $X$. One can build estimators for the Levy measure \cite{DJKP93, CohendAles95, FH06, HW12, CDGK15, FCVGCChapter} thereby estimate $H_j$ and $\widetilde{H}_j$.

A classical related problem is the channel identification, e.g., \cite{Kailath62, Bello63, DSBS10, KP05, OPZ14, PW06, PZ14}. One can show that the identifiability of the underspread WSSUS linear time variant channels can be translated for SIOs quite naturally. It is an open question worth to be investigated whether we can improve the identifiability condition for IS channels. 

A third problem is finding the fixed eigenfunctions of a linear time variant random channel $H$, e.g. \cite{KM97}. There are open directions for SIOs, especially the ones that satisfy IS property but are not wide-sense stationary, e.g. \cite{Matz05, ACN11}.

We shall provide answers to those questions in subsequent articles.

\newpage
\bibliographystyle{amsplain}
\bibliography{W.bbl}

\end{document}